\documentclass[11pt,oneside]{amsart}

      \usepackage{amssymb}
			\usepackage{geometry}
			\usepackage{amsmath,amscd}
			\usepackage{hyperref}

      \theoremstyle{plain}
      \newtheorem{theorem}{Theorem}[section]
      \newtheorem{lemma}[theorem]{Lemma}
      \newtheorem{coro}[theorem]{Corollary}

      \newtheorem{prop}{Proposition}[section]

      \newtheorem{conj}[theorem]{Conjecture}

      \theoremstyle{plain}

      \makeatletter
      \def\@setcopyright{}
      \def\serieslogo@{}
      \makeatother
   
\begin{document}

%



   \author{Adrien Boyer}
   \address{Weizmann Institute of Science, Rehovot, Israel}
   \email{adrien.boyer@weizmann.ac.il }
 \thanks{This work is supported by ERC Grant 306706. }


  
   \title[Harish-Chandra--Schwartz's algebras]{Harish-Chandra--Schwartz's algebras associated with discrete subgroups of Semisimple Lie groups}



      \subjclass[2010]{Primary  46H15, 43A90, 22E40; Secondary 22D20, 22D25, 46L80}

   \keywords{Harish-Chandra--Schwartz's spaces, semisimple Lie groups, Harish-Chandra functions, Furstenberg boundary, property RD, K-theory, Baum-Connes conjecture}

  
   \dedicatory{}



   \maketitle



\begin{abstract}
 We prove that the Harish-Chandra--Schwartz space associated with a discrete subgroup of a semisimple Lie group is a dense subalgebra of the reduced  $C^*$-algebra of the discrete subgroup. Then, we prove that for the reduced $C^*$-norm is controlled by the norm of the Harish-Chandra--Schwartz space. This inequality is weaker than property RD and holds for any discrete group acting isometrically, properly on a Riemannian symmetric space. 
\end{abstract}

\section{Introduction}

A length function on a locally compact group $G$ is a measurable function $L:G \rightarrow \mathbb{R^{*}_{+}}$  satisfying 
$L(e)=0$ where $e$ is the neutral element of the group, $L(g^{-1})=L(g)$ and $L(gh)\leq L(g)+L(h)$ for all $g,h $ in $G$. Define the Sobolev norm of parameter $d>0$ as $$\|f\|_{H_{L}^{d}(G)}:= \bigg(\int _{G}|f(g)|^{2}(1+L(g))^{2d}{\rm d }g \bigg)^{\frac{1}{2}}$$ where $f$ is in $C_{c}(G)$ the space of complex-valued continuous functions on $G$ with compact support and ${\rm d}g$ a (left) Haar measure on $G$.
The Sobolev space associated with $G$  denoted by $H_{L}^{d}(G)$ is the completion of $C_{c}(G)$ with respect to $\|\cdot \|_{H_{L}^{d}(G)}$.\\
A locally compact group $G$ has \emph{property RD (rapid decay) with respect to $L$} if there exists $d>0$ such that $H_{L}^{d}(G)$ convolves $L^{2}(G)$ in $L^{2}(G)$ in the following manner: there exists $C>0$ such that for all compactly supported functions $f$ in  $C_{c}(G)$, we have 
\begin{equation}\label{RD}
\|\lambda_{G}(f)\|_{op}\leq C \|f\|_{H_{L}^{d}(G)},
\end{equation} where $\lambda_{G}$ denotes the left regular representation of $G$, namely $\lambda_{G}(f)v=f*v$ where $f$ in $C_{c}(G)$, where $v\in L^{2}(G)$ and $\|\cdot\|_{op}$ denotes the operator norm. Notice that $\|\lambda_{G}(f)\|_{op}$ is nothing but $\|f\|_{C_{r}^{*}(G)}$ i.e. the norm of $f$ in the reduced $C^*$-algebra of $G$ denoted by $C_{r}^*(G)$.\\
 For an introduction to property RD we refer to the recent and exhaustive survey of Chatterji \cite{Ch1} and to the survey of Sapir \cite{S} for different methods of proving or disproving property RD.\\
  The semisimple Lie groups satisfy property RD  and this fact could be traced back to the
work of Herz \cite{He}  (of course he did not use this terminology) and it is essentially based on the estimates of the Harish-Chandra function \cite{Boy1}. This suggests that the essence of property RD is in harmonic analysis. For a characterization of Lie groups having property RD we refer to \cite{CPS}.\\ 
As for discrete groups, it has been first established for the free groups by Haagerup in \cite{H}. The terminology Property RD
has been introduced and studied as such by Jolissaint in \cite{J}, who notably established it for groups of polynomial growth and some hyperbolic gloups although the general case of hyperbolic groups is due to de la Harpe \cite{dlH}. In \cite{J2}, Jolissaint showed in particular that the space of rapidly decreasing functions associated with a discrete group $\Gamma$ endowed with a length function $L$, namely the space $\cap_{d>0}H_{L}^{d}(\Gamma)$, is a dense subalgebra of $C^{*}_{r}(\Gamma)$ and the inclusion induces isomorphisms in topological K-theory. \\ The first example of higher rank discrete groups having property RD is due to Ramagge, Robertson and Steger \cite{RRS}. Then, Lafforgue, inspired by the methods of \cite{RRS}, proved that cocompact lattices in $SL_{3}(\mathbb{R})$ and $SL_{3}(\mathbb{C})$ satisfy property RD. Note that Lafforgue proved in \cite{La2} that for $d$ large enough, the Sobolev space $H_{L}^{d}(\Gamma)$ associated with any discrete group $\Gamma$ is a dense Banach subalgebra of $C^{*}_{r}(\Gamma)$ and the inclusion induces isomorphisms in topological K-theory. For examples of other groups satisfying property RD we refer to \cite{BP1},  \cite{BP2},  \cite{Ch2}, \cite{Ch3}, \cite{DS}. \\
 This property  also plays a crucial role in the proof of the Novikov conjecture for hyperbolic groups by Connes and Moscovici in \cite{CM}, and Lafforgue used it to prove Baum-Connes conjecture for some groups having property (T).    The major problem related to property RD is the Valette conjecture:
\begin{conj} (The Valette conjecture)\\
Property RD holds for any discrete group acting isometrically, properly and cocompactly either on a Riemannian symmetric space or on an affine building.
\end{conj}

This problem in this generality is considered very difficult. In this paper we restrict our attention to discrete subgroups of semisimple Lie groups. We investigate another inequality involving convolutor norm of finitely supported functions and the decay of the Harish-Chandra function.  
 
 Let $G$ be a (non compact) connected semisimple real Lie group with finite center. Let $\frak{g}$ be its Lie algebra. Let $\theta$ be a Cartan involution. Define the bilinear form denoted: for all $X,Y \in \frak{g}$, $\langle X,Y \rangle=-B(X,\theta(Y))$ where $B$ is the Killing form. Set $|X|=\sqrt{\langle X,X \rangle}$. 
Write $\frak{g}=\frak{l} \oplus \frak{p}$ the eigenvector space decomposition associated to $\theta$ ($\frak{l}$ for the eigenvalue $1$). Let $K$ be the compact subgroup defined as the connected subgroup whose Lie algebra $\frak{l}$ is the set of fixed points of $\theta$. Fix $\frak{a}\subset \frak{p}$ a maximal abelian subalgebra of $\frak{p}$. Consider the root system $\Sigma$ associated to $\frak{a}$ and let $\Sigma^{+}$ be the set of positive roots, and define the corresponding positive Weyl chamber as
$$\frak{a}^{+}:=\left\{H\in \frak{a} \mbox{ such that } \alpha(H)>0, \forall \alpha \in \Sigma^{+}\right\}.$$

Let $A^{+}=\overline{\exp( \frak{a^{+})}}$ be the closure of $\exp(\frak{a^{+}})$.  Consider the corresponding Cartan decomposition $KA^{+}K$, with the Cartan projection $H:g\in G\mapsto H(g)\in \overline{\frak{a^{+}}}$. 
From now on let $L$ be the length function $$L(g)=L(ke^{H(g)}k'):=|H(g)|$$  where $g=k e^{H(g)} k'$ with $e^{H(g)}\in A^{+}$ and $k,k'\in K$. Notice that $L$ is $K$ bi-invariant.
 
Let $(G/P,\nu)$ be the Furstenberg boundary  where $P$ denotes a minimal parabolic subgroup of $G$ and $\nu$ denotes the unique $K$-invariant probability measure which is $G$-quasi-invariant. Then consider the Harish-Chandra function \begin{equation*}\label{HCH1}
\Xi:g\in G \mapsto \int_{G/P}\bigg(\frac{{\rm d}g_{*}\nu}{{\rm d}{\nu}}\bigg)^{\frac{1}{2}} (b){\rm d}\nu(b)\in \mathbb{R}^{+}.
\end{equation*}
 Define now the Harish-Chandra--Schwartz norm associated with a subgroup $H$ of $G$ as

$$\|f\|_{S^{d}_{L}(H)}:=\sup_{g\in H} \frac{|f(g)|(1+L(g))^{d}}{\Xi(g)},$$ where $f$ is in $C_{c}(H)$.
The Harish-Chandra--Schwartz space associated with $H$ denoted by $S^{d}_{L}(H)$ is defined as the completion of $C_{c}(H)$ with respect to the norm $\|\cdot \|_{S^{d}_{L}(H)}$.\\
It is a well known fact that for $G$ there exist $C,d>0$ so that for all $f$ in $C_{c}(G)$ the following inequality holds:
\begin{equation}\label{SD}
\|\lambda_{G}(f)\|_{op}\leq C \|f\|_{S^{d}_L(G)},
\end{equation}
 and we refer to \cite[Chapitre 4]{La} for proofs. 
The aim of this note is to give a proof of the following result.

\begin{theorem} \label{maintheo}
Let  $\Gamma$ be a discrete group of a connected real semisimple Lie group. Then there exist $d>0$ and $C>0$ such that

\begin{enumerate}
\item $S^{d}_{L}(\Gamma)$ is a Banach convolution algebra.
\item   $ \|\lambda_{\Gamma}(f)\|_{op}\leq C \|f\|_{S^{d}_{L}(\Gamma)} $ for all $f\in S^{d}_{L}(\Gamma)$.
\end{enumerate}
In particular $S^{d}_{L}(\Gamma)$ is a dense subalgebra of $C^{*}_{r}(\Gamma)$.
\end{theorem}

Note that this algebra has already been studied in the context of hyperbolic groups and rank one semisimple Lie groups in \cite{Boy3}.

\subsection*{Property RD and Inequality (\ref{SD}) }
One of the fundamental properties of the Harish-Chandra function is the following integrability condition $$\int_{G}\frac{\Xi^{2}(g)}{(1+L(g))^{2d}}{\rm d}g <\infty,$$
for $d>0$ large enough depending on $G$, see Subsection \ref{HCH} for more details. Moreover, it is still true that for any discrete group $\Gamma$ in the ambient semisimple Lie group $G$ we have for all $d>0$ large enough  $$\sum_{\Gamma}\frac{\Xi^{2}(\gamma)}{(1+L(\gamma))^{2d}} <\infty,$$
see \cite[Proposition 3.1]{Boy2} for further details.
 In the context of semisimple Lie groups it is clear that there exist $d,d'$ and $C>0$ such that $\|f\|_{H_{L}^{d}} \leq C  \|f\|_{S^{d'}_{L}}$ for either $f$ in $C_{c}(G)$ or in $C_{c}(\Gamma)$ of a discrete subgroup $\Gamma$ of $G$. Hence  Inequality (\ref{RD}), namely Property RD, implies Inequality (\ref{SD}) for $G$ and its discrete subgroups. The converse is wrong since all non-uniform lattices satisfy Inequality (\ref{SD}) whereas they cannot satisfy property RD.
\\

\subsection*{The Baum-Connes conjecture and the work of V. Lafforgue}
We freely rely on \cite{V}  where the reader could consult for further details about property RD, its connection with K-theory and the Baum-Connes conjecture. We refer also to the Bourbaki seminar of Skandalis \cite{Sk} for a summary of the work of Lafforgue.\\

Let $\Gamma$ be a countable group. The Baum-Connes conjecture identifies two objects associated with $\Gamma$, one analytical and one geometrical/topological.
The right-hand side of the conjecture, or analytical side, involves
the K-theory of $C^{*}_{r}(\Gamma)$. The K-theory used here denoted by $K_{i}(C^{*}_{r}(\Gamma))$ for $i=0,1$; is the usual
topological K-theory for Banach algebras.
The left-hand side of the conjecture, or geometrical/topological
side, is the $\Gamma$-equivariant K-homology with $\Gamma$-
compact supports of the classifying space $\underline{E}\Gamma$ for proper actions of $\Gamma$.

The link between both sides of the conjecture is provided by the
analytic assembly map, or index map or Baum-Connes map which is a well-defined group homomorphism:
\begin{equation}
\mu^{\Gamma}_{i}: RK^{\Gamma}_{i}
(\underline{E}\Gamma)\rightarrow K^{i}(C^{*}_{r}(\Gamma)).
\end{equation}
 The definition of the assembly map can be traced back to \cite{Ka} and to \cite{BCH}. 
 
 \begin{conj} (The Baum-Connes conjecture)\\
For $i=0,1$ the assembly map $\mu^{\Gamma}_{i}$  is an isomorphism.
\end{conj}
In the important paper \cite{La} of Lafforgue, containing breakthrough ideas about the Baum-Connes conjecture, a Banach algebra $\mathcal{A}\Gamma$ is an \emph{unconditional completion} of the group algebra $\mathbb{C}\Gamma$  if it contains $\mathbb{C}\Gamma$ as a dense subalgebra and if, for $f_{1},f_{2}$ in $\mathbb{C}\Gamma$ such that $|f_{1}(\gamma)|\leq |f_{2}(\gamma)|$ for all $\gamma$ in $\Gamma$, we have $\|f_{1}\|_{\mathcal{A} \Gamma}\leq \|f_{2}\|_{\mathcal{A}\Gamma}$. If $\mathcal{A}\Gamma$ is an unconditional completion one may define an assembly map

\begin{equation}
\mu^{\mathcal{A} \Gamma}_{i}: RK^{\Gamma}_{i}
(\underline{E}\Gamma)\rightarrow K^{i}(\mathcal{A} \Gamma).
\end{equation}
In the proof of Lafforgue of the Baum-Connes conjecture for discrete groups having property (T), one of the fundamental step is the following: 
\begin{theorem}\label{La1}(V. Lafforgue)\\
If $\Gamma$ is a group acting properly isometrically
either on a simply connected complete Riemannian manifold with non-positive curvature, bounded from below, or on a Euclidean
building then for $i=0,1$ the map $\mu^{\mathcal{A}\Gamma }_{i}: RK^{\Gamma}_{i}
(\underline{E}\Gamma)\rightarrow K^{i}(\mathcal{A}\Gamma )$ is an isomorphism.

\end{theorem}
Since the Harish-Chandra--Schwartz algebras are unconditional completions, we obtain as an immediate corollary:

\begin{coro}
If $\Gamma$ is a discrete subgroup of $G$ then for $i=0,1$ there exists $d>0$ such that the map $\mu^{S^{d}_{L}(\Gamma) }_{i}: RK^{\Gamma}_{i}
(\underline{E}\Gamma)\rightarrow K^{i}(S^{d}_{L}(\Gamma))$ is an isomorphism.
\end{coro}

After having observed that the assembly maps $\mu^{\Gamma}_{i}$ and $\mu^{\mathcal{A}\Gamma}_{i}$ are compatible, it remains to determine when one can find an embedding $\mathcal{A}\Gamma \hookrightarrow C^{*}_{r}(\Gamma)$ that induces an isomorphism of $K$-theory of $K_{i}(\mathcal{A}\Gamma)$ and $K_{i}(C^{*}_{r}(\Gamma))$. It turns out that if a discrete group $\Gamma$ satisfies property RD the inclusion of the unconditionnal algebra $H^{d}_{L}(\Gamma)$ for $d$ large enough induces isomorphism in topological K-theory and enables Lafforgue to prove the Baum-Connes conjecture for discrete group having property RD, as the cocompact lattices in rank one semisimple Lie group and in $SL_{3}(k)$ where $k$ is a local field.\\
Unfortunately, the Harish-Chandra--Schwartz algebra is not stable by holomorphic functional calculus and does not automatically induce isomorphism in topological K-theory. Consider an element $\gamma$ in $\Gamma$ so that  $\lim L(\gamma^{n})/n \rightarrow 0$. It is clear that the unit Dirac mass centered at $\gamma$ does not have the same spectral radius in $C_{r}^{*}(\Gamma)$ and in $S_{L}^{d}(\Gamma)$. \\
Nevertheless, it might be true that the embedding $S^{d}_{L}(\Gamma)\hookrightarrow C^{*}_{r}(\Gamma)$, for $d$ large enough, induces isomorphism of $K$-theory of $K_{i}(S^{d}_{L}(\Gamma))$ and $K_{i}(C^{*}_{r}(\Gamma))$. It is probably worth noting that in some cases it is true (e.g. uniform lattices in $SL(3,\mathbb{R})$), thanks to Lafforgue's work, and this is equivalent to the Baum-Connes conjecture.

 \subsection*{Structure of the paper}
 In Section \ref{sec1} we recall some standard facts about the quasi-regular representation, the Harish-Chandra function and we state a lemma due to Shalom. In Section \ref{sec2} we prove several elementary inequalities. Finally in Section \ref{sec3} we give the proof of Theorem \ref{maintheo}. 
\subsection*{Acknowledgements}
 I am grateful to Nigel Higson  who suggested that I investigate the structure of convolution algebra of Harish-Chandra--Schwartz's spaces of discrete groups and for helpful discussions concerning the work of Vincent Lafforgue and the Baum-Connes conjecture. I would like to thank Vincent Lafforgue for fruitful discussions and for pointing out that these algebras are not stable by holomorphic functional calculus. Finally, I thank Vladimir Finkelshtein for helpful comments on this paper.

\section{Preliminaries: Quasi-regular representation, the Harish-Chandra function and Shalom's lemma}\label{sec1}

\subsection{Quasi-regular representation}\label{qr}
A connected semisimple real Lie group with finite center can be written $G=KP$, where $K$ is a compact connected subgroup and $P$ is a minimal parabolic subgroup. We denote by $\Delta_{P}$ the right-modular function of $P$. The map $\Delta_{P}$ extends to $\Delta : G \rightarrow \mathbb{R^{*}_{+}}$ by $\Delta(g)=\Delta(kp):=\Delta_{P}(p)$. It is well defined because $K\cap P$ is compact (observe that ${\Delta_{P}}_{|_{K\cap P}}=1$). The quotient $G/P$, called the Furstenberg boundary, carries a unique quasi-invariant probability measure $\nu$ which is $K$-invariant, such that the Radon-Nikodym derivative at $(g,b)\in G \times G/P$ denoted by  $$c(g,b)=\frac{{\rm d}g_{*}\nu}{{\rm d}\nu}(b)=\frac{\Delta(gb)}{\Delta(b)},$$ where $g_{*}\nu(A)=\nu(g^{-1}A)$ (notice that for all $g\in G$, the function $b \in G/P \mapsto  \frac{\Delta(gb)}{\Delta(b)} \in \mathbb{R^{*}_{+}}$ is well defined). We refer to \cite[Appendix B, Lemma B.1.3, p.~344-345] {B} for more details. Let $L^{2}(G/P,\nu)$ the Hilbert space endowed with its inner product denoted by $\langle\cdot,\cdot \rangle$ and consider the quasi-regular representation  $\pi_{\nu}:G \rightarrow \mathcal{U}(L^{2}(G/P,\nu))$ associated to $P$, defined by 
\begin{equation}
(\pi_{\nu}(g)\xi)(b)=c(g,b)^{\frac{1}{2}}\xi(g^{-1}b).
\end{equation}
 Denote by ${\rm d}k$ the Haar measure on $K$, and under the identification $G/P=K/(K\cap P)$, denote by ${\rm d}[k]$ the measure $\nu$ on $G/P$. \\
 Let $\Gamma$ be a discrete subgroup of a semisimple Lie group $G$.
 In this note we consider the following unitary representation of $\Gamma$ 
 \begin{equation}
 \pi_{\nu}: \gamma \in \Gamma \mapsto \pi_{\nu}(\gamma)\in \mathcal{U}(L^{2}(G/P,\nu)),
 \end{equation}
 which is nothing but the restriction of the quasi-regular representation $\pi_{\nu}$ of $G$ to $\Gamma$.\\
Observe that if $L_{+}^{2}(G/P)$ denotes the cone of positive functions observe that $\pi_{\nu}(g)L_{+}^{2}(G/P)\subset L_{+}^{2}(G/P)$ for all $g$ in $G$.

\subsection{The Harish-Chandra function}\label{HCH}
 The well-known Harish-Chandra function can also be defined as 
 \begin{equation}
 \Xi:g\in G \mapsto \left\langle \pi_{\nu}(g)\textbf{1}_{G/P},\textbf{1}_{G/P}\right\rangle\in \mathbb{R}^{+},
 \end{equation}
  where $\textbf{1}_{G/P}$ denotes the characteristic function of the space $G/P$.\\
Observe that $\Xi$ is bi-$K$-invariant and so $\Xi(g)=\Xi(e^{H(g)})$ and $\Xi(g)=\Xi(g^{-1})$ for all $g\in G$.\\
Let $r$ be the number of indivisible positive roots in $\frak{a}$. We know that there exists $C>0$ such that for all $H \in \frak{a}$ where $e^{H} \in A^{+}$ we have $$\Xi(e^{H})\leq Ce^{-\rho(H)}\left(1+L(e^{H})\right)^{r},$$ with $\rho=\frac{1}{2}\sum_{\alpha\in \Sigma^{+}}n_{\alpha} \alpha \in \frak{a}^{+}$, see  \cite[Chap 4, Theorem 4.6.4, p.161]{GV}. Hence for $2d> \dim(\frak{a})+2r$, we set\begin{equation}\label{Cd}
\mathcal{C}_{d}:=\int_{\frak{a}^{+}} \frac{ \Xi^{2}(e^{H})}{(1+L(e^{H}))^{2d}} J(H){\rm d}H<\infty,
\end{equation}

where $J$ denotes the density that occurs in the disintegration of the Haar measure on $G$ according to the Cartan decomposition. More precisely the Haar measure ${\rm d}g$ can be written ${\rm d}g={\rm d}kJ(H){\rm d}H{\rm d}k,$ where ${\rm d}k$ is the normalized Haar measure on $K$, ${\rm d}H$ the Lebesgue measure on $\frak{a}^{+}$, and $$J(H)=\prod_{\alpha\in \Sigma^{+}}\left(\frac{ e^{\alpha(H)}-e^{-\alpha(H)}}{2} \right)^{n_{\alpha}},$$ where $n_{\alpha}$ denotes the dimension of the root space associated to $\alpha$. See  \cite[Chap.V, section 5, Proposition 5.28, p.141-142]{K}, \cite[Chap. 2, \S 2.2, p.65]{GV}  and \cite[Chap 2, Proposition 2.4.6, p.73]{GV} for more details. 
\subsection{Shalom's Lemma}
The following lemma due to Shalom \cite[Lemma 2.3]{S} will be a tool to study convolution operator norm via  the Furstenberg boundary in the semisimple Lie groups.
\begin{lemma}\label{shalom}
Let $\pi:G \rightarrow \mathcal{U}(H)$ be a unitary representation of a locally compact group. Assume that there exists a non-zero positive vector $\xi$ in the following sense:
$\langle \pi(g)\xi,\xi \rangle\geq 0$ for all $g\in G$.
Then for any bounded positive measure $\mu$ on $G$ we have
$\|\lambda_{G}(\mu)\| \leq \|\pi(\mu) \|.$
\end{lemma}
   It is obvious that $\mathbf{1}_{G/P}$ is a positive vector for $\pi_{\nu}$.\\

\section{Elementary inequalities}\label{sec2}
A function $f:G\rightarrow \mathbb{C}$ is called $L$-\emph{radial} if $f(g)=f(g')$ whenever $g$ and $g'$ satisfy $L(g)=L(g')$. 
Note that the space of $L$-radial compactly supported functions is identified by definition of $L$ to the space of compactly supported bi-K-invariant functions denoted by  $C_{c}(K \backslash G/K)$.\\
The following proposition is based on standard facts in representation theory. Nevertheless, we give a proof for the convenience of the reader. In the following, we denote by $\langle \cdot,\cdot\rangle_{L^{2}(G)}$ the scalar product of $L^{2}(G)$ with respect to the Haar measure.
 
\begin{prop}\label{radialRD}
For all $\xi,\eta$ in $L^{2}(G/P,\nu)$ and $f$ in $C_{c}(K\backslash G/K)$  we have
$$\int_{G}f(g)\langle \pi_{\nu}(g)\xi,\eta \rangle  {\rm d}g= \langle \xi,\textbf{1}_{G/P}\rangle \langle \textbf{1}_{G/P},\eta\rangle  \langle f,\Xi \rangle_{L^{2}(G)}.$$
In particular for all $d$ satisfying $2d>\dim(\frak{a})+2r$ we have for all $\xi,\eta$ in $L^{2}(G/P,\nu)$ and for all $L$-radial functions $f$ in $H^{d}_{L}( G)$
$$\bigg| \int_{G}f(g)\langle \pi_{\nu}(g)\xi,\eta \rangle  {\rm d}g  \bigg|\leq \mathcal{C}_{d}\|f\|_{H_{L}^{d}} \|\xi\|_{1}\|\eta\|_{1}.$$ 
\end{prop}
\begin{proof}
Let $\xi$ in $L^{2}(G/P,\nu)$ and observe that the function $\xi^{K}:x \in G/P \mapsto \int_{K} \xi(k^{-1}x) {\rm d}k \in \mathbb{C}$ is constant. By integration we compute its value and we have $$\xi^{K}=\langle \xi,\textbf{1}_{G/P}\rangle \textbf{1}_{G/P}.$$
 Thus, for all functions $f\in C_{c}(K\backslash G/K)$ and for all vectors $\xi,\eta\in L^{2}(G/P)$ we have:
\begin{align*}
\int_{G}f(g)&\langle \pi_{\nu}(g)\xi,\eta \rangle {\rm d}g = \int_{K}\int_{\frak{a}^{+}}\int_{K}f(ke^{H}k')\langle \pi_{\nu}(ke^{H}k')\xi,\eta \rangle {\rm d}k {\rm d}k' J(H){\rm d}H \\
&= \int_{K}\int_{\frak{a}^{+}}\int_{K}f(e^{H})\langle \pi_{\nu}(e^{H})\pi_{\nu}(k')\xi, \pi_{\nu}(k^{-1})\eta \rangle {\rm d}k {\rm d}k' J(H){\rm d}H \\
&= \int_{\frak{a}^{+}}f(e^{H})\langle  \pi_{\nu}(e^{H}) \bigg(\int_{K}\pi_{\nu}(k')\xi {\rm d}k \bigg), \bigg(\int_{K}\pi_{\nu}(k^{-1})\eta {\rm d}k \bigg) \rangle J(H){\rm d}H\\
&= \int_{\frak{a}^{+}}f(e^{H})\langle  \pi_{\nu}(e^{H}) \xi^{K},\eta^{K} \rangle J(H){\rm d}H\\
&= \langle \xi,\textbf{1}_{G/P}\rangle  \overline{\langle\eta ,\textbf{1}_{G/P} \rangle} \int_{\frak{a}^{+}}f(e^{H})\langle  \pi_{\nu}(e^{H})\mathbf{1}_{G/P},\mathbf{1}_{G/P} \rangle J(H){\rm d}H\\
&=\langle \xi,\textbf{1}_{G/P}\rangle \langle \textbf{1}_{G/P},\eta \rangle  \langle f,\Xi \rangle_{L^{2}(G)}.\\
\end{align*}
Then for all $d$ satisfying $2d>\dim(\frak{a})+2r$ we have by taking the absolute value of the above expression

\begin{align*}
\bigg| &\int_{G}f(g)\langle \pi_{\nu}(g)\xi,\eta \rangle {\rm d}g \bigg|=|\langle \xi,\textbf{1}_{G/P}\rangle|  |\langle \textbf{1}_{G/P},\eta \rangle| \bigg |\int_{\frak{a}^{+}}f(e^{H}) \Xi(e^{H}) J(H){\rm d}H \bigg |\\
&\leq\|\xi\|_{1} \|\eta\|_{1} \bigg(\int_{\frak{a}^{+}} \frac{\Xi^{2}(e^{H})}{(1+L(e^{H}))^{2d}} J(H){\rm d}H \bigg)^{\frac{1}{2}}
\bigg(\int_{\frak{a}^{+}}|f|^{2}(e^{H})(1+L(e^{H}))^{2d} J(H){\rm d}H\bigg)^{\frac{1}{2}}\\
&=\mathcal{C}_{d}\|f\|_{H_{L}^{d}}\|\xi\|_{1} \|\eta\|_{1},
\end{align*}
  where $\mathcal{C}_{d}$ is the constant defined in (\ref{Cd}) and where the last inequality follows from Cauchy-Schwarz inequality. The proof is complete.

\end{proof}

\subsection{Stability}

 Let $\Gamma$ be a discrete subgroup of $G$. A positive function $g\in G\mapsto f(g)\in \mathbb{R}^{+}$ is \emph{stable} on $G$ relative to $\Gamma$ if there exists a relatively compact neighborhood $U$ of the neutral element $e\in G$ and $C>0$ such that $$f(\gamma)\leq  Cf(g) $$  for all $\gamma \in \Gamma$ and for all $g$ in $ \gamma U$.

The next lemma gives examples of interesting stable functions.
\begin{lemma}\label{stable}
Let $\Gamma$ be a discrete group in a semisimple Lie group $G$.  Let $d$ be a real number and consider the function on $G$ defined as $$g\in G \mapsto \frac{\langle \pi_{\nu}(g)\textbf{1}_{G/P},\xi \rangle}{(1+L(g))^{d}}\in \mathbb{R}^{+},$$ with $\xi \in L_{+}^{1}(G/P,\nu)$ a positive function. Then this function is stable on $G$ relative to $\Gamma$.

\end{lemma}

\begin{proof}
Observe that the triangle inequality implies for all $\gamma \in \Gamma$ and $x\in G$ $$\frac{1}{1+L(\gamma x)}\leq \frac{1+L(x)}{(1+L(\gamma))}.$$
 Then combine the above observation with \cite[Lemma 3.8]{Boy2} to complete the proof.

\end{proof}
The notion of stability introduced above enables us to discretize inequalities given by integrals on the ambient group to sums on discrete subgroups. More precisely we have the following proposition.

\begin{prop}\label{discret}
Let $f$ be a positive $L$-radial function satisfying $  \langle f ,\Xi \rangle _{L^{2}(G)}<\infty$ and assume that $f$ is stable relative to $\Gamma$. Assume also that $\xi$ is a positive vector in $L_{+}^{1}(G/P)$. There exists a constant $C>0$ such that 
$$\sum_{\gamma}f(\gamma)\langle \pi_{\nu}(\gamma) \textbf{1}_{G/P},\xi \rangle \leq C  \langle f ,\Xi \rangle _{L^{2}(G)}\|\xi\|_{1}.$$
\end{prop}
\begin{proof}
Since $g\mapsto \langle \pi_{\nu}(g) \textbf{1}_{G/P},\xi \rangle$ is stable by Proposition \ref{stable}, there exists  $U$ a neighbourhood of the identity $e$, small enough so that $\gamma U\cap \gamma' U=\varnothing$ for all $\gamma\neq \gamma' \in \Gamma$  and a constant $C>0$ such that for all $\gamma$ in $\Gamma$ and for all $x$ in $U$ we have
 $$\langle \pi_{\nu}(\gamma) \textbf{1}_{G/P},\xi \rangle\leq C \langle \pi_{\nu}(\gamma x) \textbf{1}_{G/P},\xi \rangle\mbox{ and } f(\gamma)\leq C f(\gamma x).$$
 It follows  by integration that there exists $C>0$ such that
 \begin{align*}
f(\gamma)\langle \pi_{\nu}(\gamma) \textbf{1}_{G/P},\xi \rangle \leq &\frac{C}{{\rm Vol}(U)} \int_{U}f(\gamma x)\langle \pi_{\nu}(\gamma x) \textbf{1}_{G/P},\xi \rangle {\rm d}x,
\end{align*}
where ${\rm d}x$ is the Haar measure. 
Hence, we have 
\begin{align*}
\sum_{\gamma \in \Gamma}f(\gamma)\langle \pi_{\nu}(\gamma) \textbf{1}_{G/P},\xi \rangle &\leq \frac{C}{{\rm Vol}(U)} \sum_{\gamma \in \Gamma} \int_{U}  f(\gamma x)\langle \pi_{\nu}(\gamma x) \textbf{1}_{G/P},\xi \rangle {\rm d}x \\
&=\frac{C}{{\rm Vol}(U)} \sum_{\gamma \in \Gamma} \int_{\gamma U}  f(g)\langle \pi_{\nu}(g) \textbf{1}_{G/P},\xi \rangle {\rm d}g\\
&\leq\frac{C}{{\rm Vol}(U)}  \int_{G}  f(g)\langle \pi(g) \textbf{1}_{G/P},\xi \rangle {\rm d}g\\
&\leq\frac{C}{{\rm Vol}(U)}\langle f,\Xi \rangle_{L^{2}(G)}  \|\xi\|_{1},
\end{align*}
where the last inequality follows from Proposition \ref{radialRD}. 
\end{proof}

\subsection{Inequalities}
We start by stating and proving an elementary lemma using Cauchy-Schwarz inequality.
\begin{lemma}\label{CS}For all $\xi,\eta \in L^{2}(G/P,\nu)$ we have:
 $$|\langle \pi_{\nu}(g)\xi,\eta \rangle|\leq  \bigg(\langle\pi_{\nu}(g)\textbf{1}_{G/P},|\eta|^{2} \rangle\bigg) ^{\frac{1}{2}} \bigg( \langle \pi_{\nu}(g^{-1}) \textbf{1}_{G/P},|\xi|^{2} \rangle\bigg)^{\frac{1}{2}}. $$
\end{lemma}

\begin{proof}
Since $\pi_{\nu}$ preserves the cone of positive functions, we may assume that $\xi$ and $\eta$ are in $L_{+}^{2}(G/P)$. Then the Cauchy-Schwarz inequality implies
	\begin{align*}
	 \left\langle \pi_\nu(g)\xi,\eta\right\rangle^2
	&=\left(\int_{G/P}\xi(g^{-1}x)
	c(g,x)^{\frac{1}{2}}\eta(x)d\nu(x)\right)^2\\
	&\leq \bigg(\int_{G/P}\xi(g^{-1}x)^2
	c(g,x)^{\frac{1}{2}}d\nu(x)\bigg)\bigg(
	\int_{G/P}
	c(g,x)^{\frac{1}{2}}\eta(x)^2d\nu(x)\bigg)\\
	&=\left\langle \pi_\nu(g)\xi^2,\textbf{1}_{G/P}\right\rangle
	\left\langle \pi_\nu(g)\textbf{1}_{G/P},\eta^2\right\rangle\\
	&=	 \left\langle \pi_\nu(g^{-1})\textbf{1}_{G/P},\xi^2\right\rangle \left\langle \pi_\nu(g)\textbf{1}_{G/P},\eta^2\right\rangle
.
	\end{align*}

\end{proof}

\section{Proofs}\label{sec3}

Before proceeding to the proof of Theorem \ref{maintheo}, let us define for all $d>0$ the function $\phi_{d}$ as  \begin{equation}\phi_{d}:g\in G \mapsto  \frac{\Xi(g)}{(1+L(g))^{d}}\in \mathbb{R}^{+}.\end{equation} Observe that for all $d>0$, the function $\phi_{d}$ is a bi-K-invariant function and observe that whenever $2d> \dim(\frak{a})+2r$, we have 
\begin{equation}
\langle \phi_{2d},\Xi \rangle_{L^{2}(G)}=\mathcal{C}_{d}<+\infty.
\end{equation} Observe also that Proposition \ref{stable} implies that $\phi_{d}$ is stable on $G$ relative to any discrete subgroup $\Gamma$ of $G$. This function will play a major role in the proof. \\

\begin{proof}[Proof of Item (1) of Theorem \ref{maintheo}]
 Pick $d>0$ such that $2d> \dim(\frak{a})+2r$ and let $f_{1},f_{2}$ be in $S_{L}^{2d}(\Gamma)$. For all $g\in \Gamma$ we have
\begin{align*}
|f_{1}* f_{2}|(g)&=\sum_{\gamma \in \Gamma} f_{1}(\gamma)f_{2}(\gamma^{-1}g)\\
&\leq  \|f_{1}\|_{S_{L}^{2d}(\Gamma)} \|f_{2}\|_{S_{L}^{2d}(\Gamma)} \sum_{\gamma \in \Gamma} \frac{\Xi(\gamma)}{(1+L(\gamma))^{2d}}\frac{\Xi(\gamma^{-1}g)}{(1+L(\gamma^{-1}g))^{2d}}.
\end{align*}
Write now
\begin{align*}
\sum_{\gamma \in \Gamma} &\frac{\Xi(\gamma)}{(1+L(\gamma))^{2d}}\frac{\Xi(\gamma^{-1}g)}{(1+L(\gamma^{-1}g))^{2d}} =\\
&\sum_{ \left\{\gamma| L(\gamma)\leq \frac{L(g)}{2} \right\}} \frac{\Xi(\gamma)}{(1+L(\gamma))^{2d}}\frac{\Xi(\gamma^{-1}g)}{(1+L(\gamma^{-1}g))^{2d}} +\sum_{ \left\{\gamma| L(\gamma)> \frac{L(g)}{2} \right\}} \frac{\Xi(\gamma)}{(1+L(\gamma))^{2d}}\frac{\Xi(\gamma^{-1}g)}{(1+L(\gamma^{-1}g))^{2d}}.
\end{align*}

Observe that $ L(\gamma)\leq \frac{L(g)}{2} \Rightarrow L(\gamma^{-1}g)\geq \frac{L(g)}{2}$.		
							
We have for the first term:
\begin{align*}
\sum_{ \left\{\gamma| L(\gamma)\leq \frac{L(g)}{2} \right\}} \frac{\Xi(\gamma)}{(1+L(\gamma))^{2d}}\frac{\Xi(\gamma^{-1}g)}{(1+L(\gamma^{-1}g))^{2d}} &\leq\frac{2^{2d}}{(1+L(g))^{2d}} \sum_{ \left\{\gamma| L(\gamma)\leq \frac{L(g)}{2} \right\}} \frac{\Xi(\gamma)\Xi(\gamma^{-1}g)}{(1+L(\gamma))^{2d}} \\
& \leq\frac{ 2^{2d}}{(1+L(g))^{2d}} \sum_{ \gamma \in \Gamma} \frac{\Xi(\gamma) \Xi(\gamma^{-1}g)}{(1+L(\gamma))^{2d}}\\
& =\frac{ 2^{2d}}{(1+L(g))^{2d}} \sum_{ \gamma \in \Gamma}\phi_{2d}(\gamma) \langle \pi_{\nu}(\gamma)\textbf{1}_{G/P},\xi \rangle,
\end{align*}			
with 	$\xi:= \pi_{\nu}(g)\textbf{1}_{G/P}\in L^{2}(G/P,\nu).$
					
We have for the second term:
\begin{align*}
\sum_{ \left\{\gamma| L(\gamma)> \frac{L(g)}{2} \right\}} \frac{\Xi(\gamma)}{(1+L(\gamma))^{2d}}\frac{\Xi(\gamma^{-1}g)}{(1+L(\gamma^{-1}g))^{2d}}
&\leq \frac{2^{2d}}{(1+L(g))^{2d}}  \sum_{ \left\{\gamma| L(\gamma)> \frac{L(g)}{2} \right\}} \frac{\Xi(\gamma)\Xi(\gamma^{-1}g)}{(1+L(\gamma^{-1}g))^{2d}}\\
& \leq \frac{2^{2d}}{(1+L(g))^{2d}}  \sum_{ \gamma \in \Gamma } \frac{\Xi(\gamma)\Xi(\gamma^{-1}g)}{(1+L(\gamma^{-1}g))^{2d}}\\
&= \frac{2^{2d}}{(1+L(g))^{2d}}\sum_{ \gamma \in \Gamma } \frac{\Xi(g\gamma^{-1})\Xi(\gamma)}{(1+L(\gamma))^{2d}}\\
&= \frac{2^{2d}}{(1+L(g))^{2d}}\sum_{ \gamma \in \Gamma } \phi_{2d}(\gamma) \langle \pi_{\nu}(\gamma)\textbf{1}_{G/P},\eta \rangle
\end{align*}							
where the first equality comes from the change of variable $\gamma'=\gamma^{-1}g$ and where $\eta:=\pi_{\nu}(g^{-1})\textbf{1}_{G/P}\in L^{2}(G/P,\nu).$	\\								
			
Since $\phi_{d}$ is a stable bi-K-invariant function, Proposition \ref{discret} combined with the fact $\|\pi_{\nu}(g) \textbf{1}_{G/P}\|_{1}=\Xi(g)=\|\pi_{\nu}(g^{-1}) \textbf{1}_{G/P}\|_{1}$ imply for the first term that there exists $C>0$ so that for all $g\in \Gamma$ $$\sum_{ \left\{\gamma| L(\gamma)\leq \frac{L(g)}{2} \right\}} \frac{\Xi(\gamma)}{(1+L(\gamma))^{2d}}\frac{\Xi(\gamma^{-1}g)}{(1+L(\gamma^{-1}g))^{2d}} \leq C\frac{2^{2d}}{(1+L(g))^{2d}}\mathcal{C}_{d}\Xi(g).$$ A similar inequality holds for the second term. Therefore there exist $d,C>0$ such that for all $g\in \Gamma$ $$|f_{1}* f_{2}|(g)\leq  C \times \mathcal{C}_{d}\|f_{1}\|_{S_{L}^{2d}(\Gamma)} \|f_{2}\|_{S_{L}^{2d}(\Gamma)} \frac{\Xi(g)}{(1+L(g))^{2d}},$$ hence 
$$\| f_{1}* f_{2}\|_{S_{L}^{2d}(\Gamma)} \leq C \times \mathcal{C}_{d}  \|f_{1}\|_{S_{L}^{2d}(\Gamma)} \|f_{2}\|_{S_{L}^{2d}(\Gamma)} $$
and the proof is complete.

\end{proof}

\begin{proof}[Proof of Item (2) of Theorem \ref{maintheo}]

It is enough to prove that there exist $C,d>0$ so that for all finitely supported positive functions $f$ on $\Gamma$ $$\|\lambda_{\Gamma}(f)\|\leq C \|f\|_{S_{L}^{d}(\Gamma)}.$$
By Shalom's Lemma \ref{shalom} it is sufficient to prove 
  that there exist $C,d>0$ so that for all finitely supported positive functions $f$ on $\Gamma$ $$\|\pi_{\nu}(f)\|\leq C \|f\|_{S_{L}^{d}(\Gamma)}.$$
  Fix $d>0$ such that $2d> \dim(\frak{a})+2r$ and let $f$ be a positive function in $S_{L}^{2d}(\Gamma)$. We have for all positive vectors $\xi\in L_{+}^{1}(G/P,\nu)$ $$\sum_{\gamma}f(\gamma)\langle \pi_{\nu}(\gamma)\textbf{1}_{G/P},\xi \rangle \leq \|f\|_{S^{2d}_{L}(\Gamma)} \sum_{\gamma}\frac{\Xi(\gamma)}{(1+L(\gamma))^{2d}}\langle \pi_{\nu}(\gamma)\textbf{1}_{G/P},\xi \rangle.$$ 
  Notice also that the change of variable $\gamma'=\gamma^{-1}$ combined with the symmetry properties of $\Xi$ and $L$ imply $$\sum_{\gamma}f(\gamma)\langle \pi_{\nu}(\gamma^{-1})\textbf{1}_{G/P},\xi \rangle \leq \|f\|_{S^{2d}_{L}(\Gamma)} \sum_{\gamma}\frac{\Xi(\gamma)}{(1+L(\gamma))^{2d}}\langle \pi_{\nu}(\gamma)\textbf{1}_{G/P},\xi \rangle.$$

  Therefore Proposition \ref{discret} implies that there exist $d,C>0$ such that  for all $f\in S^{2d}_{L}(\Gamma)$ and for all positive vectors $\xi \in L_{+}^{1}(G/P,\nu)$
 \begin{equation}\label{discrete}
 \sum_{\gamma}f(\gamma)\langle \pi_{\nu}(\gamma)\textbf{1}_{G/P},\xi \rangle \leq C\times \mathcal{C}_{d} \| f\|_{S^{2d}_{L}(\Gamma)}\|\xi\|_{1}.
 \end{equation}

We have for all $\xi$ and $\eta$ in $L^{2}(G/P)$
\begin{align*}
|\langle \pi_{\nu}(f)\xi,\eta \rangle| &\leq \sum_{\gamma \in \Gamma}|f(\gamma)| \langle \pi_{\nu}(\gamma) |\xi|,|\eta|\rangle \\
&\leq \sum_{\gamma \in \Gamma}|f(\gamma)|  \big(\langle \pi_{\nu}(\gamma) \textbf{1}_{G/P},|\eta|^{2}\rangle\big)^{\frac{1}{2}} \big(\langle \pi_{\nu}(\gamma^{-1}) \textbf{1}_{G/P},|\xi|^{2}\rangle\big)^{\frac{1}{2}}\\  
&\leq \bigg( \sum_{\gamma \in \Gamma}|f(\gamma)| \langle \pi_{\nu}(\gamma) \textbf{1}_{G/P},|\eta|^{2}\rangle\bigg)^{\frac{1}{2}}  \bigg( \sum_{\gamma \in \Gamma}|f(\gamma)| \langle \pi_{\nu}(\gamma^{-1}) \textbf{1}_{G/P},|\xi|^{2}\rangle\bigg) ^{\frac{1}{2}},
\end{align*}
where the second inequality follows from Lemma \ref{CS} and the last inequality follows from Cauchy-Schwarz inequality.\\
Then Equation (\ref{discrete}) implies the existence of $C>0$ such that
\begin{align*}
|\langle \pi_{\nu}(f)\xi,\eta \rangle|&\leq (C \| f\|_{S^{2d}_{L}(\Gamma)}\mathcal{C}_{d}\|\xi\|^{2}_{2})^{\frac{1}{2}}(C \| f\|_{S^{2d}_{L}(\Gamma)} \mathcal{C}_{d}\|\eta\|^{2}_{2})^{\frac{1}{2}}\\
&=C\times \mathcal{C}_{d}\| f\|_{S^{2d}_{L}(\Gamma)} \|\xi\|_{2}\|\eta\|_{2}.
\end{align*}
By the definition of $\|\cdot\|_{op}$ we have fund $d,C>0$ such that for all finitely supported functions $f$ on $\Gamma$ we have $$\|\pi_{\nu}(f)\|\leq C \|f\|_{S_{L}^{d}(\Gamma)},$$ and the proof is done.

  \end{proof}

\end{document}